\documentclass[11pt,leqno,oneside]{amsart}
\usepackage{mathrsfs,dsfont}
\usepackage{amsmath,amstext,amsthm,amssymb,bbm,color,mathtools,comment}
\usepackage{charter}
\usepackage{typearea}
\usepackage{pdfsync}

\usepackage[width=6.5in,height=8.65in]
{geometry}

\def\bt{\begin{thm}}
\def\et{\end{thm}}
\def\bl{\begin{lem}}
\def\el{\end{lem}}
\def\bd{\begin{defi}}
\def\ed{\end{defi}}
\def\bc{\begin{cor}}
\def\ec{\end{cor}}
\def\bp{\begin{proof}}
\def\ep{\end{proof}}
\def\br{\begin{rem}}
\def\er{\end{rem}}

\newtheorem{thm}{Theorem}[section]
\newtheorem{prop}[thm]{Proposition}
\newtheorem{lem}[thm]{Lemma}
\newtheorem{defn}[thm]{Definition}

\newtheorem{rem}[thm]{Remark}
\newtheorem{cor}[thm]{Corollary}

\numberwithin{equation}{section}

\newcommand{\pk}{\Bbb{P}^k}
\newcommand{\pN}{\Bbb{P}^N}
\newcommand{\M}{\mathcal{M}}
\newcommand{\Md}{\mathcal{M}_d}
\newcommand{\Hd}{\mathcal{H}_d}

\newcommand{\la}{\langle}
\newcommand{\ra}{\rangle}

\newcommand{\om}{\omega_{FS}}

\newcommand{\f}{\textbf{f}}
\newcommand{\p}{\mu_{\textbf{f}}}

\newcommand{\E}{\Bbb{E}}
\newcommand{\B}{\mathscr{B}}

\title[ASIP for non-autonomous holomorphic dynamics in $\pk$]{Almost Sure Invariance Principle for\\ non-autonomous holomorphic dynamics in $\pk$}

\author{Turgay Bayraktar}
\date{\today}

\address{Faculty of Engineering and Natural Sciences, Sabanc{\i} University, \.{I}stanbul, Turkey}
\email{tbayraktar@sabanciuniv.edu}


\begin{document}

\begin{abstract}
We prove almost sure invariance principle, a strong form of approximation by Brownian motion, for non-autonomous holomorphic dynamical systems on complex projective space $\pk$ for H\"{o}lder continuous and DSH observables.
\end{abstract}
\maketitle
\section{Introduction}
Let $f:\pk \to \pk$ be a holomorphic map of algebraic degree $d\geq2$ and $\omega_{FS}$ denote the Fubini-Study form on $\pk$ normalized by $\int\omega_{FS}^k=1.$ Dynamical \textit{Green current} $T_f$ of $f$ is defined to be the weak limit of the sequence of smooth forms $\{d^{-n}(f^n)^*\omega_{FS}\}$ (\cite{Br,HP,FS2}). Green currents play an important role in the dynamical study of holomorphic endomorphisms of the projective space \cite{FS2,Si}. The current $T_f$ has H\"{o}lder continuous quasi-potentials, hence by Bedford-Taylor theory the exterior products $$T_f^p=T_f\wedge\dots \wedge T_f$$ are also well-defined for $1\leq p\leq k$ and dynamically interesting currents. In particular, the top degree intersection $\mu_f=T_f^k$ yields the unique $f$-invariant measure of maximal entropy (\cite{Lu,BrDu2}) with many interesting stochastic properties. For instance, in \cite{Dup} Dupont obtained an almost sure invariance principle (ASIP) for the holomorphic dynamical system $(\pk,\B,f,\mu_f)$  by using coding techniques and applying Philipp-Stout's theorem \cite{PS75} for observables with analytic singularities. The coding techniques were originally introduced by Przytycki-Urba\'nski-Zdunik \cite{PUZ} in complex dimension one from which they deduced ASIP (see also \cite{Haydn,PRL} and the references therein for some statistical results in
the case of dimension one). Recall that ASIP in the context of a holomorphic dynamical system indicates that for suitable class of observables $\psi,$ the partial Birkhoff sums $\frac1n\sum_{j=0}^{n-1}\psi\circ f^n$ can be approximated by a Brownian motion (at integer times) in such a way that almost surely the error between the trajectories is negligible relative to the size of the trajectories (see \S\ref{asip} for details). Some important statistical laws such as Central Limit Theorem (CLT) and Law of Iterated Logarithm (LIL) are among the immediate consequences of ASIP. In this context, CLT was also obtained by \cite{ClB,DS06, DNS}  by means of different methods. The approach of Dinh-Nguyen-Sibony \cite{DNS} is based on Gordin's martingale approximation method \cite{Gordin} and exponential decay of correlations for the H\"{o}lder continuous and DSH observables. Recall that a DSH function can be locally written as a difference of two plurisubharmonic (psh) functions.

\par In \cite{Bay}, we studied ergodic and statistical properties of random dynamical systems of holomorphic endomorphisms of $\pk.$ The latter is defined as successive iterations of holomorphic maps which are chosen at random according to a fixed probability measure. Under a mild assumption on the probability law, there is a naturally associated ergodic skew-product and a stationary homogenous Markov chain in that dynamical setting for which we established two versions of CLT. The first one is for the partial Birkhoff sums of the skew product for H\"{o}lder continuous and DSH observables. The second one is for the backward images of randomly chosen points with respect to the stationary probability law of the Markov chain. 

\par In the present paper, we consider non-autonomous holomorphic dynamical systems of the complex projective space $\pk.$ The latter is defined as successive iterations of a sequence of non-linear holomorphic maps of the same algebraic degree. In this setting,  we obtain more refined statistical results. Namely, we prove an ASIP (Theorem \ref{th1}) for these non-stationary systems under a natural assumptions (\ref{A}) and (\ref{B}) on the distance of the tail of the sequence from the complement of holomorphic maps. It should be emphasized that the martingale approximation method of Gordin \cite{Gordin} gives rise to a reverse martingale increment sequence. Then using the martingale CLT, one can deduce CLT for the observations of a stationary dynamical system. The ASIP can often also be obtained
in this way for the class of systems which are closed under time-reversal (see \cite{MeNi} and references therein). However, for classes of non-stationary systems that are intrinsically time-orientated, such as non-autonomous holomorphic dynamics considered in in this paper, the situation can be more delicate (see eg. \cite{MeNi,CuMe,HNTV,KKM} for examples of real dynamical systems). Here, we utilize exponential decay of correlations for DSH and H\"{o}lder continuous observables and the abstract invariance principle obtained by \cite{CuMe} in order to prove ASIP in the present setting. 

\subsection{Main Result} Recall that the set of holomorphic endomorphisms $\Hd$ of fixed algebraic degree $d$ is a Zariski open subset of the set of meromorphic maps $\Md$ and the complement $\M:=\Md\setminus\Hd$ is an irreducible hypersurface \cite{GKZ}. The set $\Md$ can be identified with $\pN$ where $N=(k+1) {{d+k}\choose{d}} -1.$ In what follows, we let $dist(\cdot,\cdot)$ denote the distance induced by the Fubini-Study metric which is normalized here such that $dist(\cdot,\cdot)\leq 1.$

\par For a sequence $\f:=(f_0,f_1,\dots)\ \text{of holomorphic maps} \ f_n:\pk\to \pk$ of fixed algebraic degree $d\geq2$ we study statistical properties of the iterates
$$F_{n}:=f_{n-1}\circ \dots \circ f_{1} \circ f_{0}:\pk\to \pk.$$  In \cite{Bay} (see also \cite{dTh1}), for such a sequence $\f=(f_j)\subset \Hd$ whose ``tail" is sufficiently far from the complement $\mathcal{M},$ we proved the existence of a measure $\mu_{\f}$ which describes asymptotic distribution of pre-images of a generic point $z\in \pk.$ 

  \begin{thm}\label{basic}\cite{Bay}
  Let $\f=(f_j)_{j\geq0}\subset \Hd$ be a sequence of holomorphic maps verifying
\begin{equation}\label{A}
\liminf_{n\to \infty}\frac1n\sum_{j=0}^{n-1}\log dist(f_j,\mathcal{M})>-\infty
\end{equation} 
and

\begin{equation}\label{B}
\liminf_{n\to \infty}\frac1j\log dist(f_j,\mathcal{M})=0.
\end{equation}   
Then there exists a probability measure $\mu_{\f}$ such that for every smooth probability  measure $\nu$ on $\pk$ 
$$\frac{1}{d^{nk}}F_{n}^*\nu \to \mu_{\f}$$ in the sense of measures as $n\to \infty.$ Moreover, the measure $\mu_{\f}$ has H\"{o}lder continuous potentials. 
  \end{thm}
  We remark that since $dist(\cdot,\cdot)\leq 1$ condition (\ref{B}) is equivalent to $\displaystyle \lim_{j\to\infty}\frac1j\log dist(f_j,\mathcal{M})=0.$ Note that any bounded sequence (i.e. away from $\mathcal{M}$) falls in the framework of Theorem \ref{basic}.  
  
In \S \ref{EME}, we provide an alternative proof for H\"{o}lder continuity of local potentials of invariant measures $\mu_{\f}$ given by Theorem \ref{basic}. Furthermore, in Theorem \ref{mix} we obtain strong mixing properties of sequential holomorphic dynamical system $(\pk,\B,\f,\p).$ As a consequence, we prove that the system $(\pk,\B,\f,\p)$ is exact, a strong form of ergodicity (Theorem \ref{exactthm}).

In what follows for a given sequence $\f=(f_j)_{j\geq0}\subset \Hd$ we let $\f_n:=(f_{n},f_{n+1},\dots)$ and denote the associated measure defined in Theorem \ref{basic} by $\mu_n.$  Our main result gives an Almost Sure Invariance Principle (ASIP) for sequences of holomorphic endomorphisms verifying hypotheses (\ref{A}) \& (\ref{B}) and for sequences of dsh and H\"{o}lder continuous observables. 
\begin{thm}[ASIP]\label{th1}
Let $\f=(f_0,f_1,\dots)$ be a sequence of holomorphic maps of fixed algebraic degree $d\geq 2$ verifying (\ref{A}) and (\ref{B}). Let also $\psi_n$ be a sequence of DSH (respectively, $\alpha$-H\"{o}lder continuous) functions satisfying $\sup_n\|\psi_n\|_{DSH}<\infty$ (respectively, $\sup_n\|\psi_n\|_{\mathscr{C}_{\alpha}}<\infty$). Assume that $\int_{\pk}\psi_nd\mu_{n}=0$ and the variances $\sigma_n^2:=\int_{\pk}(\sum_{j=0}^{n-1}\psi_j\circ F_j)^2d\mu_{\f}$ satisfy
\begin{equation}\label{variance}
Cn^{\frac14+\epsilon}\leq \sigma_n\ \ \text{for some}\ C,\epsilon>0.
\end{equation}
Then on an extended probability space there exists a sequence of centered independent Gaussian random variables $(Z_j)_{j\in \Bbb{N}}$ such that $\sum_{j=0}^{\infty}\E[Z_j^2]=\sigma_n^2(1+o(1))$ and that 
\begin{equation}\label{as}
|\sum_{j=0}^{n-1}\psi_j\circ F_j-\sum_{j=0}^{n-1}Z_j|= o(\sigma_n^{\gamma})\ \ \text{almost surely}
\end{equation}
where $\frac{1+2\epsilon}{1+4\epsilon}<\gamma<1.$
\end{thm}
 
 Some remarks are in order. The assumption $\int_{\pk}\psi_nd\mu_{n}=0$ is not restrictive as we may replace $\psi_n$ with $\psi_n-\int_{\pk}\psi_nd\mu_n.$ An immediate consequence of Theorem \ref{th1} is CLT that is 
 \begin{equation}
 \frac{1}{\sigma_n}\sum_{j=0}^{n-1}\psi_j\circ F_j\to\mathcal{N}(0,1)
 \end{equation}
 in distribution $\p$ (cf. \cite[\S 1]{PS75}). Another consequence is LIL which implies that 
 \begin{equation*}
 \limsup_{n\to\infty}[\frac{1}{\sigma_n\sqrt{2\log\log(\sigma_n)}}\sum_{j=0}^{n-1}\psi_j\circ F_j]=1 \ \ \text{almost surely}
 \end{equation*}
Finally, we remark that for an autonomous holomorphic dynamical system (i.e. $f=f_j\ \forall j$) and a single observable (i.e. $\psi=\psi_j,\ \forall j$) we have either $\sigma_n$ is bounded (in this case $\psi$ is a coboundary i.e. $\psi=\zeta-\zeta\circ f$ for some $\zeta\in L^2_{\mu_f}$) or the $\sigma_n$'s is of order $\sqrt{n}$. In this case, Theorem \ref{th1} implies that
\begin{equation}
|\sum_{j=0}^{n-1}\psi\circ f^j-\sum_{j=0}^{n-1}Z_j|=o(n^{\gamma})\  \text{almost surely}
\end{equation}
where $\frac38<\gamma<\frac12.$
 Hence, we obtain a slightly improved version of \cite[Theorem C]{Dup} as Dupont's result does not give a lower bound for the exponent $\gamma$.

I would like to thank to Referee whose comments improved the presentation of this paper.


\section{Ergodicity and Mixing for Non-autonomous systems} \label{EME}

\subsection{Invariant measures}\label{im}
 
 \par For a sequence of endomorphisms $\f=(f_j)_{j\geq0}\subset \Hd$ 
 we define \textit{topological Lyapunov exponent}
$$\chi_{top}(\f):=\lim_{n\to\infty}\frac1n\log\sup_{x\in \pk}\|D_xF_{j}\|$$ 
where $\|\cdot\|$ is a fixed norm. Clearly, this definition does not depend 
on the choice of the norm. Note that the limit exists (possibly infinite) due to
 the  sub-multiplicity of the sequence $\sup_{x\in \pk}\|D_xF_{j}\|.$ The next lemma 
will be useful in the sequel:

\begin{lem}
Let $\f=(f_j)_{j\geq0}\subset \Hd$ verifying (\ref{A}). Then $0<\chi_{top}(\f)<\infty.$
\end{lem}
\begin{proof}
Note that 
$$\Psi:\Md\times \pk \to\pk$$
$$\Psi(g,x)=g(x)$$ 
defines a meromorphic map. Then it follows from \cite[Lemma 2.1]{DD} that there exists $C>0$ 
and $q\geq 1$ such that
$$\|D_xf_j\|\leq \|D_{f_j,x}\Psi\|\leq Cdist(f_j,\M)^{-q}, \ \forall j\geq 0.$$
Then by chain rule
\begin{equation}\label{mn}
M_n:=\sup_{x\in\pk}\|D_xF_n\|\leq C^n\prod_{j=0}^{n-1}dist(f_j,\M)^{-q}
\end{equation}
and this in turn implies that
$$\frac1n\log(\sup_{x\in\pk}\|D_xF_n\|)\leq C-\frac{q}{n}\sum_{j=0}^{n-1}\log dist(f_j,\M).$$
\end{proof}
The following result was motivated by \cite[Remark 1.7.2]{Si}.
\begin{thm}\label{thm2:1}
Let $\f=(f_j)_{j\geq0}\subset \Hd$ verifying (\ref{A}) and (\ref{B}). Then
$$\frac{1}{d^n}F_n^*\om\to \om+dd^cg_{\f}$$ 
in the sense of currents. Moreover, Green function $g_f$ is  $\alpha$-H\"{o}lder continuous for every $\alpha<\frac{\log d}{\chi_{top}(\f)}.$
\end{thm}
\begin{proof}
First, we sketch the proof of existence of the limit. Let $u_j$ be a smooth qpsh function defined by
$$\frac1d f_j^*\om=\om+dd^cu_j.$$ Then
\begin{eqnarray}\label{green1}
\frac{1}{d^n}F_n^*\om &=&\frac{1}{d^n}f_0^*\cdots f_{n-1}^*\om\\
&=& \om+dd^c\sum_{j=0}^{n-1}\frac{1}{d^j}u_j\circ F_j \nonumber\\
&=:& \om+dd^cg_n \nonumber
\end{eqnarray}
By \cite[Proposition 3]{dTh1} there exists $C>0$ and $q\geq 1$ such that
\begin{equation}\label{c1}
\|u_j\|_{\mathcal{C}^1}\leq Cdist(f_j,\M)^{-q},\ \forall j\in\Bbb{N}.
\end{equation}
By (\ref{B}) for each $\epsilon>0$ there exists $j_0\in\Bbb{N}$ such that
\begin{equation}\label{bound}
dist(f_j,\M)\geq e^{-\epsilon j},\ \forall j\geq j_0.
\end{equation}

Now,
\begin{eqnarray*}
\sup_{x\in \pk}|g_n(x)-g_m(x)| &=&| \sum_{j=n}^{m-1}\frac{1}{d^j} u_j\circ F_j|\\
&\leq& \sum_{n+1}^{m-1}\frac{\sup_{x\in \pk}|u_j(x)|}{d^j}\\
&\leq & C\sum_{j=n}^{m-1}\frac{e^{\epsilon q j}}{d^j}.
\end{eqnarray*}
This imples that $g_n\to g_{\f}$ uniformly on $\pk$ for some continuous qpsh function $g_{\f}.$ 

In order to prove H\"{o}lder continuity, let $0<\alpha<\frac{\log d}{\chi_{top}(\f)}$. Note that by (\ref{mn})
$$dist(F_n(x),F_n(y))\leq C^nM_ndist(x,y).$$
Then by (\ref{c1}) and (\ref{bound}) for small $\epsilon>0$ we obtain
\begin{eqnarray*}
|g_{\f}(x)-g_{\f}(y)| &\leq& \sum_{j\geq 0}\frac{|u_j(F_j(x))-u_j(F_j(y))|}{d^j}\\
&\leq &\sum_{j\geq 0}\frac{Cdist(f_j,\M)^{-q}dist(F_j(x),F_j(y))^{\alpha}}{d^j}\\
&\leq & Cdist(x,y)^{\alpha}\sum_{j\geq 0}\frac{dist(f_j,\M)^{-q}C^{\alpha j}e^{\chi_{top}(\f)\alpha j}}{d^j}\\
&\leq & C_1dist(x,y)^{\alpha}.
\end{eqnarray*}
\end{proof}
 Since the \textit{dynamical Green} current $T_{\f}$ given by Theorem \ref{thm2:1} has H\"{o}lder continuous potentials by Bedford-Taylor theory the exterior powers
 $$T^p_{\f}:=T_{\f}\wedge\dots \wedge T_{\f}$$ are also well-defined positive closed $(p,p)$ currents for each $1\leq p\leq k.$ Letting $T_{\f_n}$ to be the dynamical Green current associated with $\f_n:=(f_{n},f_{n+1},\dots)$ by (\ref{green1}) it is easy to see that they inherit the invariance properties
 \begin{equation}
 f_{n-1}^*T_{\f_n}=d\ T_{\f_{n-1}}\ \text{and}\ (f_{n-1})_*T_{\f_{n-1}}=d^{k-1}T_{\f_n}\ \ \forall n\geq 1.
 \end{equation}
 The next result is a direct consequence of  \cite[Proposition 11]{dTh1} and Theorem \ref{basic}:
 \begin{prop}
 The top degree dynamical Green current 
 $$T^k_{\f}=T_{\f}\wedge\dots \wedge T_{\f}$$ coincides with the measure $\mu_{\f}$ given by Theorem \ref{basic}. In particular, $\mu_{\f}$ has continuous H\"{o}lder potentials and
\begin{equation}\label{invariance}
f_{n-1}^*\mu_n=d^k\mu_{n-1}\ \text{and}\ (f_{n-1})_*\mu_{n-1}=\mu_n\ \ \forall n\geq 1.
\end{equation}
 \end{prop}
\subsection{Sequential  Ergodicity, Mixing and Exactness}\label{ee} 

In what follows we let $\B$ denote the Borel algebra on $\pk$ and $(\pk,\B,\f,\p)$ be a \textit{sequential holomorphic dynamical system} that is $\f=(f_0,f_1,\dots)$ is a sequence of holomorphic maps $f_j\in\Hd$ satisfying (\ref{A}) and (\ref{B}) and $\p$ be the measure given by Theorem \ref{basic} . We say that the system $(\pk,\B,\f,\p)$ is \textit{mean ergodic} if 
\begin{equation}\label{e1}
\lim_{n\to \infty}\frac1n\sum_{j=0}^{n-1}[\p(B\cap F_j^{-1}(A))-\p(B)\mu_j(A)]=0
\end{equation}
for all $\B$-measurable sets $A,B$. Note that (\ref{e1}) is equivalent to
\begin{equation}\label{e2}
\lim_{n\to \infty}\frac1n\sum_{j=0}^{n-1}[\int_{\pk}\varphi\cdot \psi\circ F_jd\p\ -\int_{\pk}\varphi d\p \int_{\pk}\psi \mu_j] =0
\end{equation}
for all continuous (equivalently smooth, bounded or $L^2_{\p}$) functions $\varphi,\psi$ on $\pk.$ Using the the property $P_j1=1$ and the argument in \cite[Remark 1.3] {CR07} we see that (\ref{e1}) is equivalent to convergence in $L^q$-norm for $1\leq q<\infty:$

\begin{equation}\label{e3}
\lim_{n\to \infty}\|\frac1n\sum_{j=0}^{n-1}\psi\circ F_j-\int \psi d\mu_j \|_{L^q_{\p}}\to 0.
\end{equation}
Finally, we remark that in the setting of autonomous dynamical systems the mean ergodicity is equivalent to Birkhoff's ergodic theorem to which we refer here as point-wise ergodic. However, for non-autonomous systems these two concepts are not equivalent (cf. \cite{CR07}). We say that the sequential holomorphic dynamical system $(\pk,\B,\f,\p)$ is \textit{point-wise ergodic} if
\begin{equation}\label{e4}
\lim_{n\to \infty}\frac1n \sum_{j=0}^{n-1}[\psi\circ F_j(x)-\int \psi d \mu_j]=0\ \ \p-a.s.
\end{equation}
for all $L^2_{\p}$ functions $\psi$ on $\pk.$
\par We say that the system  $(\pk,\B,\f,\p)$ is \textit{mixing} if for all smooth functions $\varphi,\psi$ on $\pk$ we have 
\begin{equation}
\int_{\pk}\varphi\cdot \psi\circ F_n d\p\ -\int \varphi d\p \int\psi d\mu_n\to 0.
\end{equation} 
In Theorem \ref{mix}, we prove that every sequential holomorphic dynamical system $(\pk,\B,\f,\p)$ is mixing and hence mean ergodic. In fact, we obtain a stronger form of mean ergodicity.  

\par For a given sequence $\f=(f_j)_{j\geq0}$ of holomorphic maps $f_j\in \Hd$ we define
\begin{equation}\label{filtration}
\B_j:=F_j^{-1}(\B)=f_0^{-1}\cdots f_{j-1}^{-1}(\B).
\end{equation}
Note that
$$\B=\B_0\supset \B_1\supset\cdots\supset \B_n.$$
We also denote the asymptotic $\sigma$-algebra
$$\B_{\infty}:=\cap_{j=0}^{\infty}\B_j.$$

\begin{defn}
A sequential holomorphic dynamical system $(\pk,\B,\f,\p)$ satisfying (\ref{A}) and (\ref{B}) is called \textit{exact} if the asymptotic $\sigma$-algebra 
$$\B_{\infty}=\{\emptyset,\pk\}$$ modulo sets of $\p$-measure zero.
\end{defn}
It follows from the definition of exact sequence that $\f$ is exact if and only if
$$\cap_{j=0}^{\infty}L^2_{\mu_{\f}}(\B_j)=\{constant\ functions\}$$
where
$$L^2_{\mu_{\f}}(\B_j)=\{\xi\in L^2_{\mu_{\f}}(\pk): \xi\ \text{is} \ \B_j\ \text{measurable}\}.$$

\par In what follows we let $\f=(f_j)_{j\geq0}\subset \Hd$ satisfying (\ref{A}) and (\ref{B}).
Then each $f_j\in \Hd$ induces a unitary operator 
$$T_{j}:L^2_{\mu_{j+1}}(\pk)\to L^2_{\mu_j}(\pk)$$
$$T_j(\varphi)= \varphi\circ f_j.$$
We denote the adjoint of this operator by 
$$P_j: L^2_{\mu_j}(\pk) \to  L^2_{\mu_{j+1}}(\pk)$$
$$P_j\psi(x)=d^{-k}\sum_{f_j(y)=x}\psi(y)$$
and let
$$\mathcal{P}_j=P_{j-1}\cdots P_1P_0. $$

\begin{thm} \label{exactthm}
Let $\f=(f_j)_{j\geq0}$ be a sequence of holomorphic maps in $\Hd$ satisfying (\ref{A}) and (\ref{B}). Then the system $(\pk,\B,\f,\p)$ is exact.
\end{thm}

We remark that Theorem \ref{exactthm} generalizes \cite{Peters} in which exactness was obtained for bounded sequences $\f=(f_j)$ i.e.
$$dist(f_j,\M)\geq C>0, \ \forall j\geq 0.$$
The proof given in \cite{Peters} is based on showing that for each $F_j$ there are sufficiently many inverse branches on a disc away from the critical values for which the pre-images has small diameter. The latter argument is originally due to Briend and Duval \cite{BrDu2}. This method breaks down in the present setting and we provide a different approach.

\begin{proof}
 Note that another equivalent condition for exactness of the sequence $\f=(f_j)_{j\geq0}$ is 
\begin{equation}\label{exact}
\|\mathcal{P}_j\psi\|_{L^1_{\mu_j}}\to 0\ \text{as} \ j\to \infty
\end{equation}
 for every $\psi\in L^1_{\mu_{\f}}$ such that $\int_{\pk}\psi d\p=0.$
Indeed, given such $\psi$ it follows from Doob's martingale convergence theorem that the conditional expectations 
$$\E[\psi|\B_j]\to \E[\psi|\B_{\infty}]\ \text{in}\ L^1_{\mu_{\f}}$$ as $j\to \infty.$
Since $\E[\psi|\B_j]=(\mathcal{P}_j\psi)\circ F_j$ (cf. \cite[\S 5]{Bay}) by invariance properties (\ref{invariance}) we have
$$\|\E[\psi|\B_j]\|_{L^1_{\p}}=\|\mathcal{P}_j\psi\|_{L^1_{\mu_j}}$$ and the claim follows.

Finally, as $P_j1\equiv1\ \forall j,$ it is enough to verify the condition (\ref{exact}) for smooth functions and this follows from Lemma \ref{compact} below.
\end{proof}
Clearly, exactness implies mixing hence, mean ergodicity (\ref{e2}). Indeed, let $\varphi, \psi$ be smooth functions, we may assume that $\int_{\pk} \psi d\p=0.$ Then by (\ref{exact})
$$|\int_{\pk}\psi\cdot\varphi\circ F_n\ d\p|=|\int_{\pk}\mathcal{P}_n(\psi)\cdot\varphi\ d\mu_n|\leq C_{\varphi} \|\mathcal{P}_n\psi\|_{L^1_{\mu_n}}\to 0.$$ 
\subsubsection*{Sequential  Mixing}\label{mm}
In this section we explore mixing properties of the sequential holomorphic dynamical system $(\pk,\B,\f,\p).$ We start with some preliminaries:

\subsubsection{DSH Functions}
Recall that a function $\varphi$ is call a \textit{quasi-plurisubharmonic} (qpsh for short) if $\varphi$ can be locally written as sum of a smooth function and a pluri-subharmonic (psh) function. In what follows we denote by $L^1(\pk)$ where the norm is given by Fubini-Study volume form. We say that a function $\psi\in L^1(\pk)$ is  dsh if outside a pluripolar set $\psi=\varphi_1-\varphi_2$ where $\varphi_i$  are qpsh functions. This implies that $$dd^c\psi=T^+-T^-$$ for some positive closed $(1,1)$ currents $T^{\pm}$. Two dsh functions are identified if they coincide outside a pluripolar set; we denote the set of all dsh functions by $DSH(\pk)$. Note that dsh functions are stable under pull-back and push-forward operators induced by meromorphic self-maps of $\pk$ and have good compactness properties inherited from those of qpsh functions. Following \cite{DS3} one can define a norm on $DSH(\pk)$ as follows:
$$\|\psi\|_{DSH}:=\|\psi\|_{L^1(\pk)}+\inf\|T^{\pm}\|$$
where $dd^c\psi=T^+-T^-$ and the infimum is taken over all such representations. In what follows, $\lesssim$ and $\gtrsim$ denote inequalities up to a multiplicative constant. We remark that the currents $T^{\pm}$ have the same mass $\|T^{\pm}\|:=\int_{\pk}T^{\pm}\wedge \omega_{FS}$ as they are cohomologous and $\|\cdot\|_{DSH}\lesssim\|\cdot\|_{\mathscr{C}^2}.$ Moreover, it follows from properties of psh functions that $\|\cdot\|_{L^p}\lesssim \|\cdot\|_{DSH}$ for $1\leq p<\infty.$\\ \indent
 If $\mu$ is a probability measure on $\pk$ such that all qpsh functions are $\mu$-integrable then one can define 
 $$\|\psi\|_{DSH}^{\mu}:=|\la\mu,\psi\ra|+\inf\|T^{\pm}\|$$
where $T^{\pm}$ as above and $\la\mu,\psi\ra:=\int_{\pk}\psi d\mu$.
\begin{prop}\cite{DS3}\label{pro}
Let $\psi\in DSH(\pk)$ then there exists negative qpsh functions $\varphi_1,\varphi_2$ such that $\psi=\varphi_1-\varphi_2$ and $dd^c\varphi_i\geq -c\|\psi\|_{DSH}\omega_{FS}$ where $c>0$ independent of $\psi$ and $\varphi_i$'s. Moreover, $|\psi|$ is also a dsh function and $\||\psi|\|_{DSH}\leq c \|\psi\|_{DSH}.$ 
\end{prop} 

 Now, we turn our attention to strong mixing properties of sequential holomorphic dynamical systems. 
 
 \begin{thm}\label{mix}
   Let $(\pk,\B,\f,\p)$ be a sequential holomorphic dynamical system satisfying (\ref{A}) and (\ref{B}). Then for every $\varphi \in L^p_{\mu_n}(\pk)$ and $\psi \in DSH( \pk)$ we have 
$$ |\la \p, (\varphi\circ f_{n-1}\circ \dots \circ f_0) \psi\ra - \la\mu_n, \varphi\ra \la \p,\psi\ra| \leq Cd^{-n}||\varphi||_{L^p_{\mu_n}} \ ||\psi||_{DSH}^{\p}$$ 
where $C>0$ depends only on $\f$ and $p>1.$ Moreover, for each $0\leq \alpha\leq 2$  there exists $C>0$ depending only on $\f$ such that
$$ |\la \p, (\varphi\circ f_{n-1}\circ \dots \circ f_0) \psi\ra - \la\mu_n, \varphi\ra \la \p,\psi\ra| \leq Cd^{-\frac{\alpha n}{2}}||\varphi||_{L^p_{\mu_n}} \ ||\psi||_{\mathscr{C}^{\alpha}} $$ for each $\varphi\in L^p_{\mu_{n}}$ and $\psi$ of class $\mathscr{C}^{\alpha}.$

\end{thm}
We need several preliminary lemmas to prove Theorem \ref{mix}. 
In what follows, $C_{\f}$ denotes a constant which depends only on $\f.$
\begin{lem}\label{dsh}
 There exists $C_{\f}>0$ such that 
$$\|\psi\|_{DSH}^{\mu_n}\leq C_{\f}\ \|\psi\|_{DSH}$$
 for every $n\in \Bbb{N}$ and $\psi\in DSH(\pk).$
\end{lem}
\begin{proof}
Let $\psi \in DSH(\pk)$ then by Proposition \ref{pro} there exists qpsh functions $\phi_i$ such that $\psi=\phi_1-\phi_2$ and $dd^c\phi_i \geq -c\|\psi\|_{DSH}\omega_{FS}$ where $c>0$ is independent of $\psi$ and $\phi_i$. Since the measures $\mu_n$ have H\"{o}lder continuous super potentials (cf. \cite[Theorem1.1]{Bay}) with H\"{o}lder exponent $0<\alpha\leq 1,$ by \cite[ Lemma 3.3]{DN12} we obtain 
$$|\la \mu_n,\psi\ra|\leq C_{\f}\max(\|\psi\|_{L^1(\pk)},c^{1-\alpha}\|\psi\|_{DSH}^{1-\alpha}\|\psi\|_{L^1}^{\alpha})$$
 If $\|\psi\|_{L^1(\pk)}\geq c^{1-\alpha}\|\psi\|_{DSH}^{1-\alpha}\|\psi\|_{L^1}^{\alpha}$ we are done. Otherwise $\|\psi\|_{L^1(\pk)}<c\|\psi\|_{DSH}$ and this implies that 
$$|\la\mu_n,\psi\ra|\leq cC_{\f}\|\psi\|_{DSH}.$$
Thus, the assertion follows from Proposition \ref{pro}.
\end{proof}
\begin{rem}\label{rem}
By using a similar argument and using Lemma \ref{dsh} one can also show that there exists a constant $C_{\f}>0$  such that 
$$\|\psi\|_{DSH}\leq C_{\f}\|\psi\|_{DSH}^{\mu_n}.$$
for every $n\geq 0$ (cf. \cite[Proposition 8]{dTh1}). 
\end{rem}
The following lemma is essentially due to \cite{DNS}, however, we need to make some modifications to adapt it in our setting.
\begin{lem}\label{compact}
Let $1\leq q<\infty$ and $\psi\in DSH(\pk)$ satisfying $$\la \p,\psi\ra=0$$ then
\begin{equation}\label{strong}
\| \mathcal{P}_{n}(\psi)\|_{L^q_{\mu_j}}\leq qC_{\f}\ d^{-n} \|\psi\|_{DSH}^{\p},\ \forall j,n\geq 1
\end{equation}
where $C_{\f}>0$ depends only on the sequence $\f=(f_j)_{j\geq0}$.
\end{lem}
\begin{proof}
Note that $f_{n-1}^*\mu_{n}=d^k\mu_{n-1}$ for $n\geq1$. This implies that 
$$\la\mu_{n},P_{n-1}(P_{n-2}\circ \dots \circ P_0\psi)\ra=0.$$ Moreover, 
$$\|P_{n-1}(P_{n-2}\circ \dots \circ P_0\psi)\|_{DSH}^{\mu_{n}}\leq d^{-1}\|P_{n-2}\circ \dots \circ P_0\psi\|_{DSH}^{\mu_{n-1}}$$ 
Indeed, we may write $$dd^c(P_{n-2}\circ \dots \circ P_0\psi)=R_{n-2}^+-R_{n-2}^-$$ where $R_{n-2}^{\pm}$ are some positive closed $(1,1)$ currents.   
Then
\begin{equation}\label{normm}
\|P_{n-1}(P_{n-2}\circ \dots \circ P_0\psi)\|^{\mu_{n}}_{DSH}\leq\|d^{-k}(f_{n-1})_*(R_{n-2}^{\pm})\|= d^{-1}\|R_{n-2}^{\pm}\|
\end{equation} where the last equality follows from cohomological computation. Now by Proposition \ref{pro}, Remark \ref{rem}, Lemma \ref{dsh} and (\ref{normm}) we obtain 
\begin{eqnarray*}
\| |P_{n-1}\circ \dots \circ P_1\circ P_0(\psi)|\|_{DSH} & \leq & C \| P_{n-1}\circ \dots \circ P_1\circ P_0(\psi)\|_{DSH} \\
& \leq & C_1  \| P_{n-1}\circ \dots \circ P_1\circ P_0(\psi)\|_{DSH}^{\mu_{n}} \\
& \leq & C_2 d^{-n}\|\psi\|_{DSH}^{\p}\\
& \leq & C_3d^{-n} \|\psi\|_{DSH}
\end{eqnarray*}
where $C_3>0$ depends on $\f$ but does not depend on $n$ nor $\psi.$ Thus, by above estimate and Lemma \ref{dsh} it is enough to prove the case $\|\psi\|_{DSH}>0$. Since $$\frac{d^n}{\|\psi\|_{DSH}}|P_{n-1}\circ \dots \circ P_1\circ P_0(\psi)|$$ is a bounded sequence in $DSH(\pk),$ by Theorem \ref{basic} and \cite[Corollary 1.2]{DNS} (see also \cite[Proposition 4.4]{DN12}) there exists $\gamma>0$ and $C_{\f}>0$ independent of $\psi$ such that
$$\la\mu_{j}, \exp(\gamma \frac{d^n}{\|\psi\|_{DSH}}|P_{n-1}\circ \dots \circ P_1\circ P_0(\psi)|)\ra \leq C_{\f}$$
for all $j,n\geq 1.$
Finally, by using the inequality $\frac{x^q}{q!}\leq e^x$ for $x\geq 0$ we conclude that
$$ \| P_{n-1}\circ \dots \circ P_1\circ P_0(\psi)\|_{L^q_{\mu_{j}}}\leq q^qC_{\f,\gamma}d^{-n}\|\psi\|_{DSH}. $$ 
\end{proof}

In the autonomous case, as a consequence of interpolation theory between the Banach spaces $\mathscr{C}^0$ and $\mathscr{C}^2$ \cite{Triebel}; it was observed in \cite{DNS} that a holomorphic map $f\in\Hd$ posses strong mixing property for $\alpha$-H\"{o}lder continuous functions with $0<\alpha\leq 1$ (see \cite[Proposition 3.5]{DNS}). Adapting their argument to our setting, we obtain the succeeding lemma. We omit the proof as it is similar to the one given in Lemma \ref{compact} and to that of \cite[Proposition 3.5]{DNS}.
\begin{lem}\label{hholder}
Let $1\leq q<\infty$ and $0<\alpha\leq1$ be fixed. If $\psi:\pk\to \Bbb{R}$ be a $\alpha$-H\"older continuous function satisfying $\la\p,\psi\ra=0$ then there exists a constant $C_{\f,\alpha}>0$ independent of $\psi$ such that
$$ \| \mathcal{P}_n(\psi)\|_{L^q(\mu_{n})}\leq C_{\f,\alpha}q^{\frac{\alpha}{2}} d^{-\frac{n\alpha}{2}} \|\psi\|_{\mathscr{C}^{\alpha}}
$$
for every $n\geq1.$
\end{lem}

\begin{proof}[Proof of Theorem \ref{mix}]
If $\psi$ is constant then the assertion follows from the invariance properties $$ (f_j)_*\mu_{j}=\mu_{j+1}.$$ Thus, replacing $\psi$ by  $\psi-\la\p,\psi\ra$ we may assume that $\la\p,\psi\ra=0.$ Then by H\"older's inequality and applying Lemma \ref{compact} with $q=\frac{p}{p-1}$ we obtain
\begin{eqnarray*}
|\la \p, (\varphi\circ f_{n-1}\circ \dots \circ f_0) \psi\ra| & = & d^{-kn}|\la F_{n}^*\mu_{n}, (\varphi\circ f_{n-1}\circ \dots \circ f_0) \psi\ra| \\
& \leq & |\la \mu_{n},\varphi P_{n-1}\circ \dots \circ P_1\circ P_0(\psi)\ra| \\
& \leq & \|\varphi\|_{L^p(\mu_{n})} \| P_{n-1}\circ \dots \circ P_1\circ P_0(\psi)\|_{L^q(\mu_{n})}\\
& \leq & \frac{p}{p-1}C_{\f} d^{-n}\|\varphi\|_{L^p(\mu_{n})} \|\psi\|_{DSH}^{\p}
\end{eqnarray*}
for some $c>0$ independent of $\psi$ and for all $n\geq 0.$ 

Finally, repeating the above argument by using Lemma \ref{hholder} the second assertion follows.
\end{proof}
The following result follows from Theorem \ref{mix}; its proof is based on induction and H\"{o}lder's inequality. As the proof is similar to that of \cite[Theorem 3.4]{DNS} we omit it.
\begin{cor}\label{dshcompact}
Let $\f=(f_j)_{j\in\Bbb{N}}$ and $\mu_j$ be as in Theorem \ref{mix} and $r\geq 1$ be an integer. Further, we let $\psi_j$ be dsh functions satisfying $\sup_{0\leq j\leq r}\|\psi_j\|_{DSH}<\infty.$ Then
$$|\la \p,\psi_0(\psi_1\circ F_{n_1})\cdots (\psi_r\circ F_{n_r})\ra-\prod_{j=0}^r\la\mu_{n_j},\psi_j\ra|\leq C_{\f}\ d^{-n}\prod_{j=0}^r\|\psi_j\|_{DSH}$$
where $0=n_0\leq n_1\leq \dots \leq n_r$ and $n:=\min_{0\leq j\leq r}(n_{j+1}-n_j).$
\end{cor}
Note that we may also obtain an analogue statement of Corollary \ref{dshcompact} for H\"{o}lder continuous functions. Next, we obtain a strong law of large numbers (SLLN) for dsh and H\"{o}lder continuous observables. This result will be used to establish ASIP (Theorem \ref{th1}).

\begin{thm}\label{SLLN}
Let $\{\phi_n\}_{n\in \Bbb{N}}$ be a sequence of dsh (respectively, H\"{o}lder continuous functions with exponent $0<\alpha\leq 1$) such that $\sup_{n\in \Bbb{N}}\|\phi_n\|_{DSH}<\infty$ (respectively, $\sup_{n\in \Bbb{N}}\|\phi_n\|_{\mathscr{C}^{\alpha}}<\infty$) and $\la\mu_n,\phi_n\ra=0$ for $n\geq0.$ 
Then for each integer $r\geq 1$ and $\delta>0$ we have 
$$\lim_{n\to \infty}\frac{1}{\sqrt{n}(\log n)^{2+\delta}}\sum_{j=0}^{n-1}\big(\phi^r_j\circ F_j-\la \phi^r_j,\mu_j\ra\big)=0\ \p\text{-a.s.}$$
\end{thm} 

\begin{proof}
We prove the theorem for dsh functions as H\"{o}lder continuous case is similar. Letting
$$X_j:=\phi^r_j\circ F_j-\la \mu_j,\phi^r_j\ra\ \text{for}\ j\geq 0$$ by (\ref{invariance}) we see that $\E[X_j]:=\int_{\pk}X_jd\p=0$.
Note that by Lemma \ref{dsh} we have $\|\phi_j\|_{L^p_{\mu_j}}\lesssim \|\phi_j\|_{DSH}$ for $p\geq 1$ and we infer that
$$Var[X_j]:=\E[X_j^2]-\E[X_j]^2=\E[X^2_j]=O(1).$$ Recall that the covariance of $X_j$ and $X_{\ell}$ is given by
$$Cov(X_j,X_{\ell}):=\E[X_jX_{\ell}]-\E[X_j]\E[X_{\ell}]=\E[X_jX_{\ell}].$$
Moreover, denoting $m_j:=\la\mu_j,\phi_j\ra$ by (\ref{invariance}), triangle inequality and Corollary \ref{dshcompact} we have
\begin{eqnarray*}
|Cov(X_j,X_{j+\ell})| &\leq& |\int_{\pk}\phi^r_j\circ F_j\cdot\ \phi^r_{j+\ell}\circ F_{j+\ell}\ d\p| + |m_jm_{j+\ell}|\\
&\lesssim& d^{-\ell}\|\phi_j\|_{DSH}^r\|\phi_{j+\ell}\|_{DSH}^r + d^{-2j-\ell}\|\phi_j\|^r\|\phi_{j+\ell}\|^r \\
&\lesssim& d^{-\ell}\|\phi_j\|_{DSH}^r\|\phi_{j+\ell}\|_{DSH}^r (1+ d^{-2j}). \\
\end{eqnarray*}
where the implied constant does not depend on $\ell.$ 
\par Hence, the assertion follows from Gal-Koksma SLLN \cite[Theorem A1]{PS75}.
\end{proof}


\section{ASIP}\label{asip}
First, we recall some basic notions that we will need in the sequel. Let $(U_j)_{j\in\Bbb{N}}$ be a sequence of random variables on a probability space $(X,\mathscr{F},\mu)$ such that $\E[U_j]=0\ \forall j.$ We say that $(U_j)_{j\in\Bbb{N}}$ satisfies almost sure invariance principle (ASIP) with rates if there exists a sequence of independent centered Gaussian random variables $(Z_j)_{j\in\Bbb{N}}$ such that on an extended probability space 
\begin{equation}
\sum_{j=0}^{n-1}U_j=\sum_{j=0}^{n-1}Z_j+o(\sigma_n^{\gamma})\ \text{almost surely}
\end{equation}
where $\gamma\in (0,1)$ a fixed constant and $\sigma^2_n=\sum_{j=0}^{n-1}\E[Z_j^2]\to \infty.$

\par Recall that a Brownian motion at integer times coincides with a sum of i.i.d. Gaussian
variables, hence the above definition can also be formulated as an almost sure
approximation by a Brownian motion, with error $o(\sigma_n^{\gamma})$.

\par A sequence of random variables $(U_j)_{j\in \Bbb{N}}$ is called a \textit{reversed martingale difference} if there exists a non-increasing sequence $(\mathscr{F}_j)_{j\in \Bbb{N}}$ of $\sigma$-algebras (i.e. $\mathscr{F}_{j+1}\subset \mathscr{F}_j$) such that
\begin{itemize}
\item[(1)] $U_j$ is $\mathscr{F}_j$-measurable
\item[(2)] $\E[U_j|\mathscr{F}_{j+1}]=0\ \forall j\geq 0.$
\end{itemize} 
The next lemma will be useful in the proof of Theorem \ref{th1}.
\begin{lem}\cite[Lemma 4.2]{CuMe}\label{lemCuMe}
Let $(\zeta_n)_{n\in\Bbb{N}}$ be a sequence of reversed martingale differences in $L^p_{\mu}$ for some $1\leq p\leq 2$ with respect to a non-increasing filtration $(\mathscr{F}_j)_{j\in\Bbb{N}}.$ Assume that $$\sum_{n\geq0}\E[|\zeta_n|^p]<\infty$$ then $\sum_{n\geq0}\zeta_n$ converges $\mu$-a.s. and in $L_{\mu}^p.$ 
\end{lem}

We will use the following abstract ASIP for sequences of reversed martingale differences:

\begin{thm}\cite{CuMe}\label{CuMe}
Let $(U_j)_{j\in\Bbb{N}}$ be a sequence of square integrable random variables adapted to a non-increasing filtration $(\mathscr{F}_j)_{j\in\Bbb{N}}.$ Assume that 
$\E[U_j|\mathscr{F}_{j+1}]=0$ $\mu$-a.s and $\nu_n^2:=\sum_{j=0}^{n-1}\E[U_j^2]\to \infty$ and that $\sup_j\E[U_j^2]<\infty.$ Let also $a_n$ be a non-decreasing sequence of positive real numbers such that 
$$(a_n/\nu_n^2)_{n\in \Bbb{N}}\ \ \text{is non-increasing and}\ (a_n/\nu_n)_{n\in \Bbb{N}}\ \ \text{is non-decreasing}.$$ 
Assume further that 
\begin{equation}\label{m1}
\sum_{j=0}^{n-1}(\E[U_j^2|\mathscr{F}_{j+1}]-\E[U_j^2])=o(a_n)\ \mu-a.s.
\end{equation}

\begin{equation}\label{m2}
\sum_{j=0}^{\infty}a_n^{-q}\E[|U_j|^{2q}]<\infty\ \text{for some}\ 1\leq q\leq 2.
\end{equation}
Then on an extended probability space there exists a sequence $(Z_j)_{j\in\Bbb{N}}$ of independent centered Gaussian variables such that $\E[Z_j^2]=\E[U_j^2]$ and 
\begin{equation}
\sup_{0\leq k\leq n-1}|\sum_{j=0}^{k}U_j-\sum_{j=0}^{k}Z_j|=o(\big[a_n\big(|\log (\frac{\nu_n^2}{a_n})|+\log\log a_n\big)\big]^{\frac12})\ \ \text{almost surely}
\end{equation}
 
\end{thm}
\subsection*{Proof of Theorem \ref{th1}}
We prove the Theorem for DSH observables as H\"{o}lder continuous case is the similar. Let $(\B_j)_{j\in\Bbb{N}} $ denote the filtration defined by (\ref{filtration}). Note that 
$$\mathcal{P}_n(\psi_n\circ F_n)=\psi_n.$$
Moreover, since $\E[\phi|\B_n]$ is orthogonal projection of $\phi\in L^2(\B,\p)$ to the Hilbert space $L^2(\B_n,\p)$ it is easy to see that (cf. \cite[\S 5]{Bay})
\begin{equation}\label{3:1}
\E[\psi_n\circ F_n|\B_n]=(\mathcal{P}_n\psi_n)\circ F_n \ \ \p-a.s.
\end{equation}
This implies that 
\begin{equation}\label{3:2}
\|\E[\psi_n\circ F_n|\B_n]\|_{L^2_{\p}}=\|\mathcal{P}_n\psi_n\|_{L^2_{\mu_n}} \leq C_{\f}\ d^{-n}\|\psi_n\|^{\p}_{DSH}
\end{equation}
where the latter inequality follows from Lemma \ref{compact}.

\par Now, we define 
$$h_j=P_{j-1}(\psi_{j-1}+h_{j-1})\ \forall j\geq 1$$ and set $h_0\equiv 0.$ We also let
$$U_j:=(\psi_j+h_j-h_{j+1}\circ f_j)\circ F_j\ \forall j\geq 0.$$
Note that since $P_j(\psi_j+h_j)-h_{j+1}=0,$ by (\ref{3:1}) we obtain
$$\E[U_j|\B_{j+1}]=0$$
Hence, $(U_j)_{j\in\Bbb{N}}$ form a sequence of reversed martingale differences for the filtration $(\B_j)_{j\in\Bbb{N}} $ defined by (\ref{filtration}). Moreover,
\begin{equation}\label{3:3}
\sum_{j=0}^{n-1}U_j=S_n-h_{n}\circ F_{n}
\end{equation}
where $S_n:=\sum_{j=0}^{n-1}\psi_j\circ F_j.$
Furthermore, by invariance properties (\ref{invariance}) and Lemma \ref{compact} we have
\begin{equation}\label{3:3.25}
\|h_n\circ F_n\|_{L^2_{\p}}=\|h_n\|_{L^2_{\mu_n}}\leq C_1
\end{equation} where $C_1>0$ depends only on $\f$ but independent of $n.$ This implies that
\begin{eqnarray}\label{3:3.5}
|\|S_n\|_{L^2_{\p}}-\|\sum_{j=0}^{n-1} U_j\|_{L^2_{\p}}| &\leq & \|S_n-\sum_{j=0}^{n-1} U_j\|_{L^2_{\p}}\\
&=& \|h_n\circ F_n\|_{L^2_{\p}}\leq C_1 \nonumber
\end{eqnarray}
and we conclude that $\|\sum_{j=0}^{n-1} U_j\|^2_{L^2_{\p}}\to \infty.$ Then using $\int_{\pk}U_iU_jd\p=0$ for $i\not=j$ we deduce that $$\nu_n^2=\sum_{j=0}^{n-1}\E[U^2_j]=\|\sum_{j=0}^{n-1} U_j\|^2_{L^2_{\p}}\to \infty$$
where $\E[U^2_j]=\int_{\pk}U_j^2d\p.$ Moreover, by (\ref{3:3.5})
\begin{equation}\label{3:3.75}
|\sigma_n-\nu_n|\leq C_1.
\end{equation}
We also remark that by Lemma \ref{compact}
\begin{equation}\label{3:4}
\sup_{j}\|U_j\|_{L^q_{\p}}\leq C_q,\ \ \forall q\geq1.
\end{equation}
First, we establish ASIP for $(U_j)_{j\geq0}$ by verifying the hypotheses (\ref{m1}) and (\ref{m2}) of Theorem \ref{CuMe}. To this end let $\frac{1+\epsilon}{1+4\epsilon}<\gamma<1$ and we choose $a_n=\nu_n^{2\gamma}\gtrsim n^{\frac12+\epsilon}$ which is clearly non-decreasing sequence of  positive real numbers verifying $a_n/\nu_n^2= \nu_n^{2\gamma-2}$ is decreasing and $a_n/\nu_n=\nu_n^{2\gamma-1} \uparrow \infty$ since $\gamma>\frac12$. 
Note that by (\ref{3:4}), (\ref{3:3.75}) and the assumption (\ref{variance}) on $\sigma_n$ we obtain
\begin{equation}\label{normU}
\sum_{j\geq0}\frac{\|U_j\|^4_{L^4_{\p}}}{a_j^2}\leq \sum_{j\geq1}\frac{C_3}{j^{1+4\epsilon}}<\infty
\end{equation} which verifies (\ref{m2}) with $q=2.$ Moreover, by (\ref{normU}) and Lemma \ref{lemCuMe}  
$$\sum_{j=0}^{\infty}\frac{\E[U_j^2|\B_{j+1}]-U_j^2}{a_j}$$
converges $\p$-a.s. Then by Kronecker's Lemma \cite[Theorem 2.5.5]{Durrett}
$$\sum_{j=0}^{n-1}\big(\E[U_j^2|\B_{j+1}]-U^2_j\big)=o(a_n)\ \p\text{-a.s.}$$
Hence, in order to get (\ref{m2}), it is enough to show that 
\begin{equation}\label{claim}
\sum_{j=0}^{n-1}\big(U_j^2-\E[U_j^2]\big)=o(a_n)\ \p\text{-a.s.}
\end{equation}
Let us denote by $$\phi_j:=(\psi_j+h_j-h_{j+1}\circ f_j)$$ so that $U^2_j=\phi^2_j\circ F_j$ for $j\geq0.$ Note that $\phi_j$ are dsh functions and by Lemma \ref{compact} their dsh norms are bounded. Moreover, since $\la \mu_j,h_j\ra=0$ by (\ref{invariance}) we have $\la\mu_j,\phi_j\ra=0$ for $j\geq 0.$ Hence, the claim (\ref{claim}) follows from Theorem \ref{SLLN}.

Now, applying Theorem \ref{CuMe} for $(U_j)_{j\geq0}$ on an extended probability space we obtain centered independent Gaussian random variables $(Z_j)_{j\geq0}$ such that $\E[Z_j^2]=\E[U_j^2]$ and 

\begin{eqnarray*}
\sup_{0\leq k\leq n-1}|\sum_{j=0}^{k}U_j-\sum_{j=0}^{k}Z_j|&=&o(\sqrt{a_n(|\log(\frac{\nu_n^2}{a_n})|+\log\log a_n)})\\
&=& o(\sigma^{\gamma}_n\sqrt{\log \sigma_n})\ \ \text{almost surely}
\end{eqnarray*}
where the last equality follows from (\ref{3:3.75}).

Finally, we remark that $h_j$ are dsh functions and by Lemma \ref{compact} their dsh norms are bounded. Moreover, by (\ref{invariance}) the means $\la \mu_j,h_j\ra=0$ and by Theorem \ref{SLLN} 
$$\sum_{j=0}^{n-1}h_j\circ F_j=o(\sigma^{\gamma}_n\sqrt{\log \sigma_n})\ \ \text{almost surely}$$
Since the $\log$-term can be absorbed in $\sigma_n^{\gamma}$ when $\gamma$ varies in an open interval, (\ref{as}) follows from (\ref{3:3}), (\ref{3:3.25}). This finishes the proof for DSH observables. 

For the H\"{o}lder continuous observables, Theorem \ref{SLLN} does not apply directly to the functions $\phi_j.$ In this case, using a standard convolution and a partition of unity, we can approximate $\psi_j$ by functions $\psi_{j,\delta}$ satisfying 
$$\|\psi_{j,\delta}\|_{\mathscr{C}^2}\lesssim \delta^{-2}\ \text{and}\ \ \sup_{\pk}|\psi_{j}(x)-\psi_{j,\delta}(x)|\lesssim \delta^{r}$$ for some $r>0.$ Then we estimate the covariances as in Theorem \ref{SLLN} and obtain SLLN. The details are left to the reader. 

\qed


 \bibliographystyle{amsalpha}

\bibliography{biblio}
\end{document}